\documentclass[12pt, reqno]{amsart}
\title{Wired Cycle-Breaking Dynamics for Uniform Spanning Forests}
\author{Tom Hutchcroft}
\address{University of British Columbia}
\email{thutch@math.ubc.ca}
\date{\today}

\usepackage{amsmath,amssymb,amsfonts,amsthm}
\usepackage{color}

\usepackage{graphicx}
\usepackage[nameinlink]{cleveref}

\crefname{theorem}{Theorem}{Theorems}
\crefname{thm}{Theorem}{Theorems}
\crefname{lemma}{Lemma}{Lemmas}
\crefname{lem}{Lemma}{Lemmas}
\crefname{remark}{Remark}{Remarks}
\crefname{prop}{Proposition}{Propositions}
\crefname{defn}{Definition}{Definitions}
\crefname{corollary}{Corollary}{Corollaries}
\crefname{conjecture}{Conjecture}{Conjectures}
\crefname{question}{Question}{Questions}
\crefname{chapter}{Chapter}{Chapters}
\crefname{section}{Section}{Sections}
\crefname{figure}{Figure}{Figures}
\newcommand{\crefrangeconjunction}{--}

\theoremstyle{plain}
\newtheorem{thm}{Theorem}

\newtheorem{lem}[thm]{Lemma}
\newtheorem{corollary}[thm]{Corollary}
\newtheorem{prop}[thm]{Proposition}

\theoremstyle{definition}
\newtheorem{defn}[thm]{Definition}
\newtheorem{example}[thm]{Example}

\theoremstyle{remark}

\numberwithin{equation}{section}

\renewcommand{\P}{\mathbb P}
\newcommand{\Z}{\mathbb Z}
\newcommand{\E}{\mathbb E}

\newcommand{\N}{\mathbb N}

\usepackage[numeric,initials,nobysame]{amsrefs}

\usepackage{caption}
\usepackage{subcaption}
\usepackage{thmtools}
\usepackage{thm-restate}
\usepackage{commath}
\usepackage{mathrsfs}
\usepackage{bbm}
 \usepackage[margin=1.55in]{geometry}
\usepackage{enumitem}

\newcommand{\defo}[1]{\textbf{#1}}

\newcommand{\eqd} {\overset{d}{=}}

\newcommand{\F}{F}

\newcommand{\cF}{\mathcal F}

\newcommand{\cT}{\mathcal T}

\newcommand{\orient}[1]{#1}
\newcommand{\rev}[1]{-#1}

\newcommand{\head}[1]{#1^{{+}}}
\newcommand{\tail}[1]{#1^{{-}}}

\newcommand{\unorient}[1]{#1}

\newcommand{\USF}{\mathsf{USF}}
\newcommand{\WUSF}{{\mathsf{WUSF}}}
\newcommand{\tWUSF}{\mathsf{WUSF}}
\newcommand{\FUSF}{\mathsf{FUSF}}

\newcommand{\UST}{\mathsf{UST}}
\newcommand{\OWUSF}{\mathsf{OWUSF}}

\newcommand{\OUST}{\mathsf{OUST}}

\begin{document}

\begin{abstract}
	We prove that every component of the  wired uniform spanning forest (WUSF) is one-ended almost surely in every transient reversible random graph, removing the bounded degree hypothesis required by earlier results. We deduce that every component of the WUSF is one-ended almost surely in every supercritical Galton-Watson tree, answering a question of Benjamini, Lyons, Peres and Schramm \cite{BLPS}.

	Our proof introduces and exploits a family of Markov chains under which the oriented WUSF is stationary, which we call the \emph{wired cycle-breaking dynamics}. 
\end{abstract}

\maketitle

\section{Introduction}

The \defo{uniform spanning forests} ($\USF$s) of an infinite, locally finite, connected graph $G$ are defined as infinite-volume limits of uniformly chosen random spanning trees of large finite subgraphs of $G$. These limits can be taken with respect to two extremal boundary conditions, \defo{free} and \defo{wired}, giving the \defo{free uniform spanning forest} ($\FUSF$) and \defo{wired uniform spanning forest} ($\WUSF$) respectively (see \cref{sec:WUSF} for detailed definitions). The study of uniform spanning forests was initiated by Pemantle \cite{Pem91}, who, in addition to showing that both limits exist, proved that the wired and free forests coincide in $\Z^d$ for all $d$ and that they are almost surely a single tree if and only if $d \leq 4$. The question of connectivity of the $\WUSF$ was later given a complete answer by Benjamini, Lyons, Peres and Schramm (henceforth referred to as BLPS) in their seminal work \cite{BLPS}, in which they proved that the $\WUSF$ of a graph is connected if and only if two independent random walks on the graph intersect almost surely \cite{BLPS}*{Theorem 9.2}.

After connectivity, the most basic topological property of a forest is the number of ends its components have. An infinite graph $G$ is said to be \defo{$k$-ended} if, over all finite sets of vertices $W$, the graph $G \setminus W$ formed by deleting $W$ from $G$ has a maximum of $k$ distinct infinite connected components. In particular, an infinite tree is one-ended if and only if it does not contain any simple bi-infinite paths and is two-ended if and only if it contains a unique simple bi-infinite path.

Components of the $\tWUSF$ are known to be one-ended for several large classes of graphs. Again, this problem was first studied by Pemantle \cite{Pem91}, who proved that the $\USF$ on $\Z^d$ has one end for $2 \leq d \leq 4$ and that every component has at most two ends for $d \geq 5$. 
(For $d=1$ the forest is all of $\Z$ and is therefore two-ended.) 
A decade later, BLPS \cite{BLPS}*{Theorem 10.1} completed and extended Pemantle's result, proving in particular that every component of the $\WUSF$ of a Cayley graph is one-ended almost surely if and only if the graph is not itself two-ended.
Their proof was then adapted to random graphs by Aldous and Lyons \cite{AL07}*{Theorem 7.2}, who showed that all $\WUSF$ components are one-ended almost surely in every transient reversible random rooted graph with bounded vertex degrees. Taking a different approach, Lyons, Morris and Schramm \cite{LMS08} gave an isoperimetric condition for one-endedness, from which they deduced that all $\WUSF$ components are one-ended almost surely in every transient transitive graph and every non-amenable graph.

In this paper, we remove the bounded degree assumption from the result of Aldous and Lyons \cite{AL07}. We state our result in the natural generality of reversible random rooted networks. Recall that a \defo{network} is a locally finite, connected (multi)graph $G=(\mathsf{V},\mathsf{E})$ together with a function $c:\mathsf{E} \to (0,\infty)$ assigning a positive \defo{conductance} $c(e)$ to each unoriented edge $e$ of $G$. 
For each vertex $v$, the conductance $c(v)$ of $v$ is defined to be the sum of the conductances of the edges adjacent to $v$, where self-loops are counted twice. Graphs without specified conductances are considered to be networks by setting $c\equiv 1$.  The $\WUSF$ on a network is defined in \cref{sec:WUSF} and reversible random rooted networks are defined in \cref{sec:thm1proof}.

\begin{thm}\label{thm:one-endunimod} Let $(G,\rho)$ be a transient reversible random rooted network with $\E[c(\rho)^{-1}]<\infty$. Then every component of the wired uniform spanning forest of $G$ is one-ended almost surely. \end{thm}

The condition that the expected inverse conductance of the root is finite is always satisfied by graphs, for which $c(\rho)=\deg(\rho)\geq1$. In Example \ref{ex:counterexample} we show that the theorem can fail in the absence of this condition.

\cref{thm:one-endunimod} applies (indirectly) to supercritical Galton-Watson trees conditioned to survive, answering positively Question 15.4 of BLPS \cite{BLPS}.

\begin{corollary}\label{cor:GW} Let $T$ be a supercritical Galton-Watson tree conditioned to survive. Then every component of the wired uniform spanning forest of $T$ is one-ended almost surely. \end{corollary}
Previously, this was known only for supercritical Galton-Watson trees with  offspring distribution either bounded, in which case the result follows as a corollary to the theorem of Aldous and Lyons \cite{AL07}, or supported on a subset of $[2,\infty)$, in which case the tree is non-amenable and we may apply the theorem of Lyons, Morris and Schramm \cite{LMS08}.

Our proof introduces a new and simple method, outlined as follows. For every transient network, we define a procedure to `update an oriented forest at an edge', in which the edge is added to the forest while another edge is deleted. Updating oriented forests at randomly chosen edges defines a family of Markov chains on oriented spanning forests, which we call the \emph{wired cycle-breaking dynamics}, for which the oriented wired uniform spanning forest measure is stationary (\cref{prop:transientupdate}). This stationarity allows us to prove the following theorem, from which  we show \cref{thm:one-endunimod} to follow by known methods.
\begin{thm}\label{thm:threeends} Let $G$ be any network. If the wired uniform spanning forest of $G$ contains more than one two-ended component with positive probability, then it contains a component with three or more ends with positive probability. \end{thm} 

The case of \emph{recurrent} reversible random rooted graphs remains open, even under the assumption of bounded degree. In this case, it should be that the single tree of the $\WUSF$ has the same number of ends as the graph (this prediction appears in \cite{AL07}). BLPS proved this for transitive recurrent graphs \cite{BLPS}*{Theorem 10.6}. 



\subsection{Consequences} The one-endedness of $\WUSF$ components has consequences of fundamental importance for the \emph{Abelian sandpile model}.  J\'arai and Werning~\cite{JarWer14} proved that the infinite-volume limit of the sandpile measures exists on every graph for which every component of the $\WUSF$ is one-ended almost surely. Furthermore, J\'arai and Redig \cite{JarRed08} proved that, for any graph which is both transient and has one-ended $\WUSF$ components, the sandpile configuration obtained by adding a single grain of sand to the infinite-volume random sandpile can be stabilized by finitely many topplings (their proof is given for $\Z^d$ but extends to this setting, see \cite{Jar14}). Thus, a consequence of \cref{thm:one-endunimod} is that these properties hold for the Abelian sandpile model on transient reversible random graphs of unbounded degree.

\cref{thm:one-endunimod} also has several interesting consequences for random plane graphs, which we address in upcoming work with Angel, Nachmias and Ray \cite{unimodular2}. In particular, we deduce from \cref{thm:one-endunimod} that every Benjamini-Schramm limit of finite planar graphs is almost surely Liouville, i.e.\ does not admit non-constant bounded harmonic functions.






\section{The Wired Uniform Spanning Forest}\label{sec:WUSF}

In this section we briefly define the wired uniform spanning forest and introduce the properties that we will need. For a comprehensive treatment of uniform spanning trees and forests, as well as a detailed history of the subject, we refer the reader to Chapters 4 and 10 of \cite{LP:book}.  

\medskip
\paragraph{\textbf{Notation and orientation}} We do not distinguish notationally between oriented and unoriented trees, forests or edges. 
Whether or not a tree, forest or edge is oriented will be clear from context. Edges $e$ are oriented from their tail $\tail{e}$ to their head $\head{e}$, and have reversal $-e$. Given an oriented tree or forest in a graph,  we define the \defo{past} of each vertex $v$ to be the set of vertices $u$ for which there is a directed path from $u$ to $v$ in the oriented tree or forest.
\medskip




For a finite graph $G$, we write $\UST_G$ for the uniform measure on the set of spanning trees (i.e.\ connected cycle-free subgraphs containing every vertex) of $G$, considered for measure-theoretic purposes to be functions from $\mathsf{E}$ to $\{0,1\}$. More generally, if $G$ is a finite network, we define $\UST_G$ to be the probability measure on spanning trees of $G$ for which the measure of a tree $t$ is proportional to the product of the conductances of its edges $\prod_{e\in t} c(e)$.

There are two extremal (with respect to stochastic ordering) ways to define infinite volume limits of the uniform spanning tree measures.
Let $G$ be an infinite network and let $G_n$ be an increasing sequence of finite subnetworks (i.e.\ subgraphs of $G$ with inherited conductances) such : that $\bigcup G_n = G$, which we call an \defo{exhaustion} of $G$.  The weak limit of the $\UST_{G_n}$ is known as the \defo{free uniform spanning forest}: for each finite subset $S\subset \mathsf{E}$,
\[ \FUSF_G(S \subseteq \F) := \lim_{n \to \infty} \UST_{G_n}(S \subseteq T). \]
Alternatively, at each step of the exhaustion we define a network $G_n^*$ by identifying (`wiring') every vertex of $G\setminus G_n$ into a single vertex $\partial_n$ and deleting all the self-loops that are created, and define the \defo{wired uniform
spanning forest} to be the weak limit 
\[ \WUSF_G(S \subseteq \F) := \lim_{n\to\infty} \UST_{G_n^*}(S \subseteq T). \]
 Both limits were shown (implicitly) to exist for every network and every choice of exhaustion by Pemantle \cite{Pem91}, although the $\WUSF$ was not defined explicitly until the work of H\"aggstr\"om \cite{Hagg95}. As a consequence, the limits do not depend on the choice of exhaustion. Both measures are supported on spanning forests (i.e.\ cycle-free subgraphs containing every vertex) of $G$ for which every connected component is infinite.  The $\WUSF$ is usually much more tractable, thanks in part to Wilson's algorithm rooted at infinity, which both connects the $\WUSF$ to loop-erased random walk and allows us to sample the $\WUSF$ on an infinite network directly rather than passing to an exhaustion.


\medskip

\defo{Wilson's algorithm} \cite{Wilson96,ProppWilson} is a remarkable method of generating the $\UST$ on a finite or recurrent network by joining together loop-erased random walks. It was extended to generate the $\WUSF$ on transient networks by BLPS \cite{BLPS}.
Let $G$ be a network, and let $\gamma$ be a path in $G$ that is either finite or transient, i.e.\ visits each vertex of $G$ at most finitely many times. The \defo{loop-erasure} $\textsf{LE}( \gamma )$ is formed by erasing cycles from $\gamma$ chronologically as they are created. Formally, $\textsf{LE}(\gamma)_i  = \gamma_{t_i}$ where the times $t_i$ are defined recursively by $t_0 = 0$ and $t_i = 1+ \max \{ t \geq t_{i-1} : \gamma_t = \gamma_{t_{i-1}}\}$. 

Let $\{v_j : j \in \mathbb{N} \}$ be an enumeration of the vertices of $G$ and define a sequence of forests in $G$ as follows:
\begin{enumerate} \item If $G$ is finite or recurrent, choose a root vertex $v_0$ and let $\F_0$ include $v_0$ and no edges (in which case we call the algorithm \defo{Wilson's algorithm rooted at $v_0$}). If $G$ is transient, let $\F_0 = \emptyset$ (in which case we call the algorithm \defo{Wilson's algorithm rooted at infinity}).
\item Given $\F_i$, start an independent random walk from $v_{i+1}$ stopped if and when it hits the set of vertices already included in $\F_i$. 

\item Form the loop-erasure of this random walk path and let $\F_{i+1}$ be the union of $\F_i$ with this loop-erased path.
\item Let $F = \bigcup F_i$.
\end{enumerate} This is Wilson's algorithm: the resulting forest $\F$ has law $\UST_G$ in the finite case \cite{Wilson96} and $\WUSF_G$ in the infinite case \cite{BLPS}, and is independent of the choice of enumeration.  

We also consider oriented spanning trees and forests.
Let $\mathsf{OUST}_{G_n^*}$ denote the law of the uniform spanning tree of $G_n^*$ oriented towards the boundary vertex $\partial_n$, so that every vertex of $G_n^*$ other than $\partial_n$ has exactly one outward-pointing oriented edge in the tree. Wilson's algorithm on $G_n^*$ rooted at $\partial_n$ may be modified to produce an oriented tree with law $\OUST_{G_n^*}$ by considering the loop-erased paths in step (2) to be oriented chronologically. If $G$ is transient, making the same modification to Wilson's algorithm rooted at infinity yields a random oriented forest, known as the \defo{oriented wired uniform spanning forest} \cite{BLPS} of $G$ and denoted $\OWUSF_G$.
The proof of the correctness of Wilson's algorithm rooted at infinity  \cite{BLPS}*{Theorem 5.1} also shows that, when $G_n$ is an exhaustion of a transient network $G$,
the measures $\mathsf{OUST}_{G_n^*}$ converge weakly to $\OWUSF_G$.


\section{Wired Cycle-Breaking Dynamics}\label{sec:Markov}

 Let $G$ be an infinite transient network and let $\orient{\cF}(G)$ denote the set of oriented spanning forests $\orient{f}$ of $G$ such that every vertex has exactly one outward-pointing edge in $\orient{f}$.
For each $\orient{f} \in \orient{\cF}(G)$ and  oriented edge $\orient{e}$ of $G$, the update $U(\orient{f},\orient{e})\in\cF(G)$ of $\orient{f}$ is defined by the following procedure:

\begin{defn}[Updating $\orient{f}$ at $e$] \label{def:update}
If $\orient{e}$ or its reversal $\rev{e}$ is already included in $\orient{f}$, or is a self-loop, let $U(\orient{f},\orient{e})=\orient{f}$. Otherwise,
\begin{itemize}
\item If $\head{e}$ is in the past of $\tail{e}$ in $\orient{f}$, so that there is a directed path $\langle \orient{e_1},\ldots,\orient{e_k},\orient{d} \hspace{.1em}\rangle$ from $\head{e}$ to $\tail{e}$ in $\orient{f}$, let \[U(\orient{f},\orient{e}) =  \orient{f} \cup \{ \rev{e},\rev{e_1}, \ldots, \rev{e_k} \} \setminus \{ \orient{d},\orient{e_k},\ldots,\orient{e_1}\}.\]
\item Otherwise, if $\head{e}$ is not in the past of $\tail{e}$ in $\orient{f}$, let $d$ be the unique oriented edge of $\orient{f}$ with $\tail{d}=\tail{e}$ and let $U(\orient{f},\orient{e})=\orient{f}\cup\{\orient{e}\}\setminus\{\orient{d}\}$.
\end{itemize}
Note that in either case, \emph{as unoriented forests}, we have simply
\[\unorient{U(\orient{f},\orient{e})}=\unorient{f}\cup\{\unorient{e}\}\setminus\{\unorient{d}\}.\]
\end{defn}

Let $v$ be a vertex of $G$. We define the \textbf{wired cycle-breaking dynamics rooted at $v$} to be the Markov chain on $\orient{\cF}(G)$ with transition probabilities
\[ p^v(\orient{f_0},\orient{f_1})= \frac{1}{c(v)}c(\{\orient{e} : \tail{e}=v \text{ and } U(\orient{f_0},\orient{e})=\orient{f_1}\}).\]
That is, we perform a step of the dynamics by choosing an oriented edge randomly from the set $\{\orient{e}:\tail{e}=v\}$ with probability proportional to its conductance, and then updating at this edge. Dynamics of this form for the $\UST$ on \emph{finite} graphs are well-known, see \cite{LP:book}*{\S 4.4}.

To explain our choice of name for these dynamics, as well as our choice to consider oriented forests, let us give a second, equivalent, description of the update rule. 
\begin{quote}
If $\orient{e}$ or its reversal $\rev{e}$ is already included in $\orient{f}$, or is a self-loop, let $U(\orient{f},\orient{e})=\orient{f}$. Otherwise,
\begin{itemize}[leftmargin=*]
\item If $\head{e}$ and $\tail{e}$ are in the same component of $\orient{f}$, then $\unorient{f}\cup\unorient{e}$  contains a (not necessarily oriented)  cycle. Break this cycle by deleting the unique edge $\orient{d}$ of $f$ that is both contained in this cycle and adjacent to $\tail{e}$, letting $U(\orient{f},\orient{e})=f\cup\{e\}\setminus\{d\}$. 

If $\head{e}$ was in the past of $\tail{e}$ in $f$, so that there is an oriented path from  $\tail{e}$ to $\head{d}$ in $U(f,e)$, correct the orientation of $U(f,e)$ by reversing each edge in this path.
\item If $\head{e}$ and $\tail{e}$ are not in the same component of $\orient{f}$, we consider $e$ together with the two infinite directed  paths in $\orient{f}$ beginning at $\tail{e}$ and $\head{e}$ to constitute a \textbf{wired cycle}, or `cycle through infinity'. Break this wired cycle by deleting the unique edge $\orient{d}$ in $\orient{f}$ such that $\tail{d}=\tail{e}$, letting $U(f,e)=f\cup\{e\}\setminus\{d\}$. 
\end{itemize}
\end{quote}
The benefit of taking our forests to be oriented is that it allows us to define these wired cycles unambiguously. 
 If every component of the $\WUSF$ on $G$ is one-ended almost surely, then there is a unique infinite simple path from each of $\tail{e}$ and $\head{e}$ to infinity, so that wired cycles are already defined unambiguously and the update rule may be defined without reference to an orientation.

\begin{prop}\label{prop:transientupdate} Let $G$ be an infinite transient network. Then for each vertex $v$ of $G$, $\OWUSF_G$ is a stationary measure for the wired cycle-breaking dynamics rooted at $v$, i.e.~ for $p^v(\hspace{.1em}\cdot\hspace{.225em},\hspace{.05em}\cdot\hspace{.15em})$.\end{prop}







\begin{proof} Let $\langle G_n \rangle_{n\geq1}$ be an exhaustion of $G$ by finite networks. We may assume that $G_n$ includes $v$ and all of its neighbours for all $n\geq1$. By the Skorokhod representation theorem, there exist random variables $\langle\orient{T_n}\rangle_{n \geq 1}$ and $\orient{\F}$, defined on some common probability space, such that $\orient{T_n}$ has law $\mathsf{OUST}_{G_n^*}$ for each $n$, $\orient{\F}$ has law $\OWUSF_G$, and $\orient{T_n}$ converges to $\orient{\F}$ almost surely as $n$ tends to infinity. Let $\orient{E}$ be an oriented edge chosen randomly from the set $\{\orient{e}:\tail{e}=v\}$ with probability proportional to its conductance, independently of $\langle \orient{T_n}\rangle_{n\geq1}$ and $\orient{\F}$. We write $\P$ for the probability measure under which $\langle T_n \rangle_{n\geq1}$, $F$ and $E$ are sampled as indicated.

Let $\cT(G_n^*)$ denote the set of spanning trees of $G_n^*$ oriented towards the boundary vertex $\partial_n$. For each $t\in \cT(G_n^*)$ and oriented edge $\orient{e}$ with $\tail{e}=v$, we define the update $U(\orient{t},\orient{e})$ of $\orient{t}$ at $\orient{e}$ by the same procedure (\cref{def:update}) as for $f\in \cF(G)$. We first show that stationarity holds in the finite case.

\begin{lem} $U(T_n,E)\eqd T_n$ for every $n\geq1$. \end{lem}
\begin{proof}[Proof of lemma] Define a Markov chain on $\cT(G_n^*)$, as we did on $\cF(G)$, by
\[ p^v(\orient{t_0},\orient{t_1})= \frac{1}{c(v)}c(\{\orient{e} : \tail{e}=v \text{ and } U(\orient{t_0},\orient{e})=\orient{t_1}\}).\]
The claimed equality in distribution is equivalent to $\OUST_{G_n^*}$ being a stationary measure for $p^v(\hspace{.1em}\cdot\hspace{.225em},\hspace{.05em}\cdot\hspace{.15em})$, and so it suffices to verify that $\OUST_{G_n^*}$ satisfies the detailed balance equations for $p^v(\hspace{.1em}\cdot\hspace{.225em},\hspace{.05em}\cdot\hspace{.15em})$. That is, for every pair of distinct oriented trees $\orient{t_0},\orient{t_1} \in \cT(G_n^*)$, we show that
\[ p^v(\orient{t_0},\orient{t_1})\prod_{e \in t_0} c(e)=p^v(\orient{t_1},\orient{t_0})\prod_{e \in t_1}c(e).\] 
If $p^v(\orient{t_0},\orient{t_1})>0$, then there exists an oriented edge $\orient{g}$ with $\tail{g}=v$ such that $U(\orient{t_0},\orient{g})=\orient{t_1}$. Since the two trees are distinct, $g$ is in $t_1$ but not $t_2$, and so updating $\orient{t_0}$ at any other oriented edge pointing out of $v$ cannot give $\orient{t_1}$. Thus, 
\[ p^v(\orient{t_0},\orient{t_1}) = c(g)/c(v). \]
Let $\orient{d}$ be the edge that is deleted from $\orient{t_0}$ when updating at $\orient{g}$. 
Updating $\orient{t_1}$ at whichever of $d$ or $-d$ has $v$ as its tail gives back $\orient{t_0}$, and so
\[ p^v(\orient{t_1},\orient{t_0}) = c(d)/c(v).\]
Finally, since as unoriented trees $\unorient{t_1}=\unorient{t_0}\cup\{g\}\setminus\{d\}$, we have 
\begin{align*} p^v(\orient{t_0},\orient{t_1})\prod_{e \in t_0} c(e) = \frac{c(g)c(d)}{c(v)}\prod_{e \in \unorient{t_0}\cap\unorient{t_1}} c(e) =p^v(\orient{t_1},\orient{t_0})\prod_{e \in t_1}c(e) \end{align*}
as desired (the products are taken over unoriented edges). 
\end{proof}

To complete the proof, we show that $U(\orient{T_n},\orient{E})$ converges to $ U(\orient{\F},\orient{E})$ in probability. 


Let $r$ be a natural number, and let $R$ be the maximum of $r$ and the length of the longest finite simple path in $\F$ connecting $v$ to one of its neighbours in $G$ that is in the same component as $v$ in $F$. Since $\orient{T_n}$ converges to $\F$ almost surely, there exists a random $N$ such that $\orient{T_n}$ and $\orient{\F}$ coincide on the ball $B_R(v)$ of radius $R$ about $v$ in $G$ for all $n\geq N$. In particular, the unique outward-pointing edge from $v$ in $\orient{\F}$ is also included in $\orient{T_n}$ for all $n\geq N$.

We claim that, with probability tending to one, $\orient{\F}$ and $\orient{T_n}$ agree about whether or not $\head{E}$ is in the past of $v$.
\begin{lem}\label{lem:past}
Let $\mathscr{P}$ and $\mathscr{P}_n$ denote the events $\{\head{E} \text{ is in the past of $v$ in $\orient{\F}$}\}$ and $\{\head{E} \text{ is in the past of $v$ in $\orient{T_n}$}\}$ respectively. Then the probability of the symmetric difference $\mathscr{P}\triangle\mathscr{P}_n$ converges to zero as $n\to\infty$.
\end{lem}
\begin{proof}[Proof of lemma]
Given $\orient{E}$, the probability that $\head{E}$ is in the past of $v$ in $\orient{T_n}$ is, by Wilson's algorithm, the probability that $v$ is contained in the loop-erasure of a random walk from $\head{E}$ to $\partial_n$ in $G_n^*$. Since $G$ is transient, this probability converges to the probability that $v$ is contained in the loop-erased random walk from $\head{E}$ in $G$. This probability is exactly the probability that $\head{E}$ is in the past of $v$ in $\orient{\F}$, and so 
\begin{equation}\label{eq:past1}\nonumber \P ( \mathscr{P}_n) \xrightarrow[n \to \infty]{} \P( \mathscr{P}).\end{equation}
If $\P(\mathscr{P})\in\{0,1\}$, we are done. Otherwise, on the event $\mathscr{P}$, there is by definition a finite directed path from $\head{E}$ to $v$ in $\orient{\F}$. This directed path is also contained in $\orient{T_n}$ for all $n\geq N$ and so
\begin{equation*}\label{eq:past2}\nonumber  \P ( \mathscr{P}_n \,| 
 \mathscr{P} ) \xrightarrow[n\to\infty]{}1. \end{equation*}
Combining these two above limits gives 
\[\P(\mathscr{P}_n \,|\, \neg \mathscr{P} ) = \frac{\P(\mathscr{P}_n)-\P(\mathscr{P}_n\,|\,\mathscr{P})\P(\mathscr{P})}{\P(\neg\mathscr{P})} \xrightarrow[n \to \infty]{}0. \]
and hence
\begin{align*}\P(\mathscr{P}{\triangle}\mathscr{P}_n) &= \P(\mathscr{P})  - \P(\mathscr{P}\cap\mathscr{P}_n) + \P(\mathscr{P}_n\cap\lnot\mathscr{P})\\
&= \P(\mathscr{P})  - \P(\mathscr{P}_n \,|\,\mathscr{P})\P(\mathscr{P}) + \P(\mathscr{P}_n \,|\,\neg\mathscr{P})\P(\neg\mathscr{P})\\
& \xrightarrow[n \to\infty]{} \P(\mathscr{P}) - \P(\mathscr{P}) + 0 = 0.\end{align*}
as desired.
\end{proof}

Thus, it suffices to show that $U(\orient{\F},\orient{E})$ and $U(\orient{T_n},\orient{E})$ coincide on $B_R(v)$ on the event that $n\geq N$ and that $\orient{\F}$ and $\orient{T_n}$ agree about whether or not $\head{E}$ is in the past of $v$, since this event has probability tending to 1 as $n$ tends to infinity.

First consider the event that $\head{E}$ is not in the past of $v$ in both $\orient{\F}$ and $\orient{T_n}$ and that $n\geq N$. In this case, updating either $\orient{\F}$ or $\orient{T_n}$ at $\orient{E}$ has the identical effect of deleting the unique outward-pointing edge from $v$ (which coincides in $\orient{\F}$ and $\orient{T_n}$), adding $\orient{E}$, and leaving all orientations unchanged. So $U(\orient{\F},\orient{E})$ and $U(\orient{T_n},\orient{E})$ coincide in $B_R(v)$ on this event. 

Next consider the event that $\head{E}$ is in the past of $v$ in both $\orient{\F}$ and $\orient{T_n}$ and that $n\geq N$. Then there is a directed path $\langle \orient{d_1},\ldots,\orient{d_k}\rangle$ from $\head{E}$ to $v$ in $\orient{\F}$. By the definition of $R$, this path is also contained in $\orient{T_n}$. It follows that updating either $\orient{\F}$ or $\orient{T_n}$ at $\head{E}$ has the identical effect of deleting $\orient{d_k}$, adding $\orient{E}$, and reversing the orientation of the path $\langle \orient{E},\orient{d_1},\ldots,\orient{d_{k-1}}\rangle$. Consequently, the updates $U(\orient{\F},\orient{E})$ and $U(\orient{T_n},\orient{E})$ coincide on $B_R(v)$ on this event, completing the proof. \qedhere

\end{proof}
\subsection{Update-tolerance}
For each edge $\orient{e}$ and event $\mathscr{A}\subset\orient{\cF}(G)$, we define the event $U(\mathscr{A},\orient{e})$ to be the image of $\mathscr{A}$ under the update $U(\,\cdot\,,\orient{e})$. That is, \[ U(\mathscr{A},\orient{e}) = \{ U(\orient{f},\orient{e}) : \orient{f} \in \mathscr{A}\}. \]
An immediate consequence of \cref{prop:transientupdate} is that the probability of $U(\mathscr{A},e)$ cannot be much smaller than the probability of $\mathscr{A}$.
\begin{corollary}\label{cor:updatetolerance}Let $G$ be a transient network and let $e$ be an oriented edge of $G$. Then for every event $\mathscr{A}\subset \cF(G)$,
\[ \OWUSF_G(U(\mathscr{A},e)) \geq \frac{c(e)}{c(e^-)}\OWUSF_G(\mathscr{A}). \]
\end{corollary}

We refer to this property as \textbf{update-tolerance} by analogy to the well-established theories of insertion- and deletion-tolerant invariant percolation processes \cite{LP:book}*{Chapters 7 and 8}.

\section{Proof of \cref{thm:threeends}.}\label{sec:threeendsproof}

\begin{proof} Let $G$ be a network such that the $\WUSF$ on $G$ contains at least two two-ended connected components with positive probability. Since $G$'s $\WUSF$ is therefore disconnected with positive probability, Wilson's algorithm implies that $G$ is necessarily transient. The \defo{trunk} of a two-ended tree is defined to be the unique bi-infinite simple path contained in the tree, or equivalently the set of vertices and edges in the tree whose removal disconnects the tree into two infinite connected components. 

Let $\orient{\F}_0$ be a sample of $\OWUSF_G$. By assumption, there exists a (non-random) path $\langle\gamma_i\rangle_{i=0}^n$ in $G$ such that with positive probability,
 $\gamma_0$ and $\gamma_n$ are in distinct two-ended components of $\orient{\F}_0$, 
$\gamma_n$ is in the trunk of its component, 
and  $\gamma_i$ is not in the trunk of $\gamma_n$'s component for $i<n$.
Write $\mathscr{A}_\gamma$ for this event.

For each $1\leq i\leq n$, let $e_i$ be an edge with $\tail{e_i}=\gamma_i$ and $\head{e_i}=\gamma_{i-1}$, and let $\F_i\in\cF(G)$ be defined recursively by \[F_{i+1}=U(F_i,e_{i+1}).\]  We claim that on the event $\mathscr{A}_\gamma$, the component containing $\gamma_n$ in the updated forest $\F_n$ has at least three ends. Applying update-tolerance (\cref{cor:updatetolerance}) iteratively will then imply that the probability of the $\WUSF$ containing a component with three or more ends is at least
\[ \OWUSF_G(\mathscr{A}_\gamma)\prod_{i=1}^n\frac{c(e_i)}{c(\gamma_i)} \]
which is positive as claimed.


\begin{figure}[t]
\centerline{
 \begin{minipage}{1.1\textwidth}
 \centering
\includegraphics[width=1\textwidth]{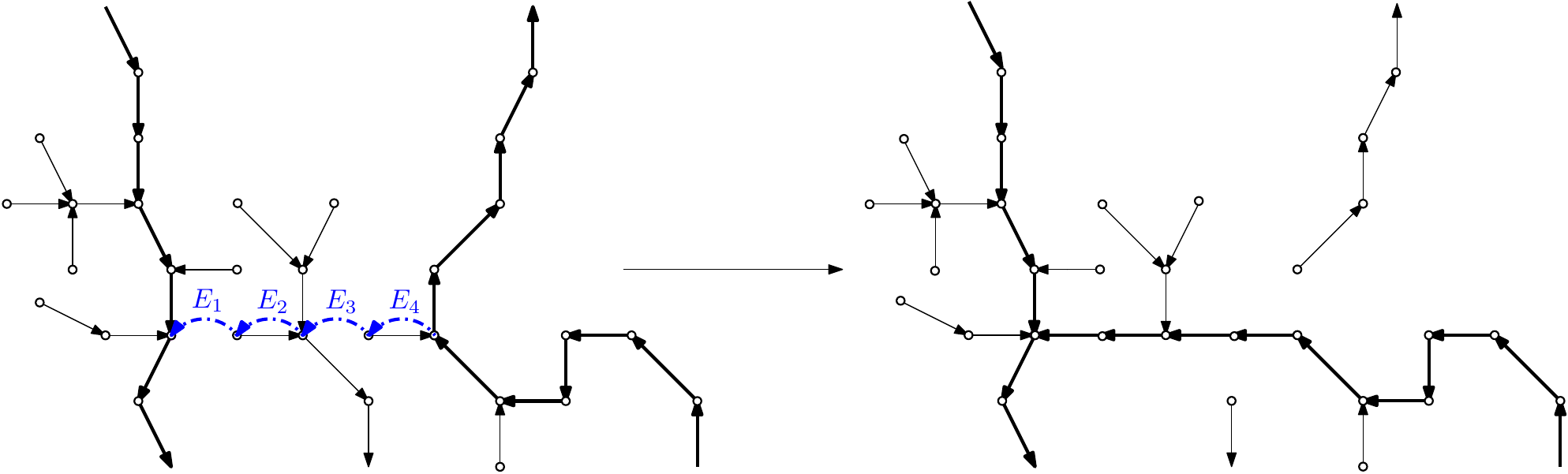}
\caption{
	\small{
		When we update along a path (blue arcs) connecting a two-ended component to the trunk of another two-ended component (with each edge oriented backwards), a three-ended component is created. Edges whose removal disconnects their component into two infinite connected components are bold.
	}
}
 \end{minipage}
 }
\end{figure}

First, notice that $\gamma_i$'s component in $\orient{\F_i}$ has at least two ends for each $0\leq i \leq n$. This may be seen by induction on $i$. The component of $\gamma_0$ in $\orient{\F_0}$ is two-ended by assumption, while for each $0\leq i < n$:
\begin{itemize}
\item If $\gamma_{i+1}$ is in the same component as $\gamma_i$ in $\orient{\F_i}$, then the component containing $\gamma_{i+1}$ in the updated forest $\orient{\F}_{i+1}$ has the same number of ends as the component of $\gamma_i$ in $\orient{\F_i}$.
\item If $\gamma_{i+1}$ is in a different component to $\gamma_i$ in $\orient{\F_i}$, then the component containing $\gamma_{i+1}$ in $\orient{\F_{i+1}}$ consists of the component of $\gamma_{i}$ in $\orient{\F_i}$, the edge $\orient{E_i}$ and the past of $\gamma_{i+1}$ in $\orient{\F_i}$. Thus, the component of $\gamma_{i+1}$ in $\orient{\F_{i+1}}$ has at least as many ends as the component of $\gamma_i$ in $\orient{\F_i}$.  
\end{itemize}
Now, $\gamma_i$ is not in the trunk of $\gamma_n$'s component in $\orient{\F_0}$ for $i<n$, and so in particular $\gamma_n$ is not in the past of $\gamma_i$ in $\F_{i-1}$ for any $i<n$. It follows that $\gamma_{n-1}$ and $\gamma_n$ are in different components of $\F_{n-1}$, and that the trunk of $\gamma_n$'s component in $\F_0$ is still contained in $\F_{n-1}$. From this, we see that $\gamma_n$'s component in $\F_n$ has at least three ends as claimed, see Figure 1.
\end{proof}

\section{Reversible random networks and the proof of Theorem 1.}\label{sec:thm1proof}
A \defo{rooted network} $(G,\rho)$ is a network $G$ together with a distinguished vertex $\rho$, the root. An isomorphism of graphs is an isomorphism of rooted networks if it preserves the conductances and the root.
A \defo{random rooted network} $(G,\rho)$ is a random variable taking values in the space of isomorphism classes of random rooted networks (see \cite{AL07} for precise definitions, including that of the topology on this space). Similarly, we define \defo{doubly-rooted networks} to be networks together with an ordered pair of distinguished vertices. Let $(G,\rho)$ be a random rooted network and let $\langle X_n\rangle_{n\geq0}$ be simple random walk on $G$ started at $\rho$. We say that $(G,\rho)$ is \defo{reversible} if the random doubly-rooted networks $(G,\rho,X_n)$ and $(G,X_n,\rho)$ have the same distribution
\[ (G,\rho,X_n) \eqd (G,X_n,\rho) \]
for every $n$, or equivalently for $n=1$. Be careful to note that this is not the same as the reversibility of the random walk on $G$, which holds for any network. Reversibility is essentialy equivalent to the related property of \defo{unimodularity}. We refer the reader to  \cite{AL07} for a systematic development and overview of the beautiful theory of reversible and unimodular random rooted graphs and networks, as well as many examples.


\medskip

We now deduce \cref{thm:one-endunimod} from \cref{thm:threeends}. Our proof that the $\WUSF$ cannot have a unique two-ended component is adapted closely from Theorem 10.3 of \cite{BLPS}.
\begin{proof}[Proof of \cref{thm:one-endunimod}]
Let $(G,\rho)$ be a reversible random rooted network such that $\E[c(\rho)^{-1}]<\infty$. Biasing the law of $(G,\rho)$ by the inverse conductance $c(\rho)^{-1}$ (that is, reweighting the law of $(G,\rho)$ by the Radon-Nikodym derivative $c(\rho)^{-1}/\E[c(\rho)^{-1}]$) gives an equivalent unimodular random rooted network, as can be seen by checking involution invariance of the biased measure \cite[Proposition 2.2]{AL07}. This allows us to apply Theorem 6.2 and Proposition 7.1 of \cite{AL07} to deduce that every component of the $\WUSF$ of $G$ has at most two ends almost surely.
\cref{thm:threeends} then implies that the $\WUSF$ of $G$ contains at most one two-ended component almost surely. 


Suppose for contradiction that the $\WUSF$ contains a single two-ended component with positive probability. Recall that the trunk of this component is defined to be the unique bi-infinite path in the component, which consists exactly of those edges and vertices whose removal disconnects the component into two infinite connected components. 

Let $\langle X_n \rangle_{n\geq0}$ be a random walk on $G$ started at $\rho$, and let $\F$ be an independent random spanning forest of $G$ with law $\WUSF_G$, so that (since $\WUSF_G$ does not depend on the choice of exhaustion of $G$) the sequence $\langle (G,X_n,F)\rangle_{n\geq0}$ is stationary. If the trunk of $F$ is at some distance $r$ from $\rho$, then $X_r$ is in the trunk with positive probability, and it follows by stationarity that $\rho$ is in the trunk of $\F$ with positive probability. We will show for contradiction that in fact the probability that the root is in the trunk must be zero.

Recall that for each $n$, the forest $\F$ may be sampled by running Wilson's algorithm rooted at infinity, starting with the vertices $\rho$ and $X_n$. If we sample $\F$ in this way and find that both $\rho$ and $X_n$ are contained in $\F$'s unique trunk, we must have had either that the random walk started from $\rho$ hit $X_n$, or that the random walk started from $X_n$ hit $\rho$. Taking a union bound,
\[\P(\rho \text{ and } X_n \text{ in trunk}) \leq \begin{array}{l}\hspace{.785em}\P(\text{random walk started at $X_n$ hits $\rho$})\\+\P(\text{random walk started at $\rho$ hits $X_n$}).\end{array} \]
By reversibility, the two terms on the right hand side are equal and hence
\begin{equation*}\P(\rho \text{ and } X_n \text{ in trunk}) \leq 2\P(\text{random walk started at $X_n$ hits $\rho$}). \end{equation*}
The probability on the right hand side is now exactly the probability that simple random walk started at $\rho$ returns to $\rho$ at time $n$ or greater, and by transience this converges to zero. Thus,
\begin{equation*}\P(\rho \text{ and } X_n \text{ in trunk}) = \E\left[\mathbbm{1}(\rho \text{ in trunk})\mathbbm{1}(X_n \text{ in trunk})\right] \xrightarrow[n \to \infty]{} 0\end{equation*}
and so
\begin{equation}\label{eq:transience}\tag{$\star$} \E\left[\mathbbm{1}(\rho \text{ in trunk})\frac{1}{n}\sum_1^n\mathbbm{1}(X_i \text{ in trunk})\right] \xrightarrow[n \to \infty]{} 0. \end{equation}
Let $\mathcal{I}$ be the invariant $\sigma$-algebra of the stationary sequence $\langle (G,X_n,F) \rangle_{n\geq0}$. The Ergodic Theorem implies that
\[ \frac{1}{n}\sum_1^n\mathbbm{1}(X_i \text{ in trunk}) \xrightarrow[n \to \infty]{\text{a.s.}} \P(\rho \text{ in trunk} \,|\, \mathcal{I}\,).\]
Finally, combining this with \eqref{eq:transience} and the Dominated Convergence Theorem gives
\[\E\left[\mathbbm{1}(\rho \text{ in trunk})\cdot \P(\rho \text{ in trunk}\,|\,\mathcal{I}\,)\right] = \E\left[\P(\rho \text{ in trunk}\,|\,\mathcal{I}\,)^2\right]=0. \]
It follows that $\P(\rho \text{ in trunk})=0$, contradicting our assumption that $\F$ had a unique two-ended component with positive probability.
\end{proof}

\begin{proof}[Proof of {\cref{cor:GW}}.]\label{ex:Galton-Watsontrees} 
 Given a probability distribution $\langle p_k ;\; k\geq0\rangle$ on $\N$, the \textbf{augmented Galton-Watson tree} $T$ with offspring distribution $\langle p_k \rangle$ is defined by taking two independent Galton-Watson trees $T_1$ and $T_2$, both with offspring distribution $\langle p_k \rangle$, and then joining them by a single edge between their roots. Lyons, Pemantle and Peres~\cite{LPP95} proved that $T$ is reversible when rooted at the root of the first tree $T_1$, see also \cite{AL07}*{Example 1.1}.

If the distribution $\langle p_k \rangle$ is supercritical (i.e.~ has expectation greater than 1), the associated Galton-Watson tree is infinite with positive probability and on this event is almost surely transient \cite{LP:book}*{Chapter 16}. Thus, \cref{thm:one-endunimod} implies that every component of $T$'s $\WUSF$ is one-ended almost surely on the event that either $T_1$ or $T_2$ is infinite.

Recall that for every graph $G$ and every edge $e$ of $G$ which has a positive probability of not being included in $G$'s $\WUSF$, the law of $G$'s $\WUSF$ conditioned not to contain $e$ is equal to $\WUSF_{G\setminus\{e\}}$ \cite{BLPS}*{Proposition 4.2}. 
Let $e$ be the edge between the roots of $T_1$ and $T_2$ that was added to form the augmented tree $T$. On the positive probability event that $T_1$ and $T_2$ are both infinite, running Wilson's algorithm on $T$ started from the roots of $T_1$ and $T_2$ shows, by transience of $T_1$ and $T_2$,  that $e$ has  positive probability not to be included in $T$'s $\WUSF$. On this event, $T$'s $\WUSF$ is distributed as the union of independent samples of $\WUSF_{T_1}$ and $\WUSF_{T_2}$. It follows that every component of $T_1$'s $\WUSF$ is one-ended almost surely on the event that $T_1$ is infinite.
\end{proof}

\begin{example}[{$\E[c(\rho)^{-1}]<\infty$} is necessary]\label{ex:counterexample}
Let $(T,o)$ be a 3-regular tree with unit conductances rooted at an arbitrary vertex $o$. Form a network $G$ by adjoining to each vertex $v$ of $T$ an infinite path, and setting the conductance of the $n$th edge in each of these paths to be $2^{-n-1}$. Let $o_n$ be the $n$th vertex in the added path at $o$. Define a random vertex $\rho$ of $G$  which is equal to $o$ with probability $4/7$ and equal to the $n$th vertex in the path at $o$ with probability $3/(7\cdot2^{n})$ for each $n\geq1$. The only possible isomorphism classes of $(G,\rho,X_1)$ are of the form $(G,o_{n},o_{n+1})$, $(G,o_{n+1},o_n)$, $(G,o,o_1)$, $(G,o_1,o)$, or $(G,o,o')$, where $o'$ is a neighbour of $o$ in $T$. This allows us to easily verify that $(G,\rho)$ is a reversible random rooted network:
\[ \P((G,\rho,X_1)=(G,o_n,o_{n+1})) = \P((G,\rho,X_1)=(G,o_{n+1},o_n)) = \frac{1}{7\cdot 2^n} 
\]
for all $n\geq1$ and
\[ \P((G,\rho,X_1)=(G,o,o_1))  = \P((G,X_1,\rho)=(G,o,o_1)) = \frac{1}{7}.\]
 

When we run Wilson's algorithm on $G$ started from a vertex of $T$, every excursion of the random walk into one of the added paths is erased almost surely. It follows that the $\WUSF$ on $G$ is simply the union of the $\WUSF$ on $T$ with each of the added paths, and hence every component has infinitely many ends almost surely.
\end{example}


\subsection*{Acknowledgements}
We thank Tel Aviv University for its hospitality while this work was completed. We also thank Omer Angel and Asaf Nachmias for many valuable comments on earlier versions of the paper.

\footnotesize{
	\bibliographystyle{abbrv}
	\bibliography{unimodular}
 }
\end{document}